\documentclass[12pt,oneside,reqno]{amsart}
\usepackage{amsmath,amsthm,amssymb,epic,eepic}
\usepackage{eqlist,eqparbox}
\usepackage{diagrams}
\usepackage{epsfig}
\usepackage{cancel}
\usepackage{comment}
\usepackage{tikz}

\newcommand{\ekviv}{\Longleftrightarrow}

\newtheorem{observation}{Observation}

\definecolor{ChadDarkBlue}{rgb}{.1,0,.2}  
\definecolor{ChadBlue}{rgb}{.1,.1,.5}  
\definecolor{ChadRoyal}{rgb}{.2,.2,.8}  
\definecolor{ChadGreen}{rgb}{0,.4,0}    
\definecolor{ChadRed}{rgb}{.5,0,.5}  
 
\def\zruseno#1{} 

\def\logoesf{%
\begin{tabular}{l l}
\begin{tabular}{c}
{Supported by}\\
\phantom{\huge X}
\end{tabular}& \ \resizebox{8.58cm}{!}{\includegraphics{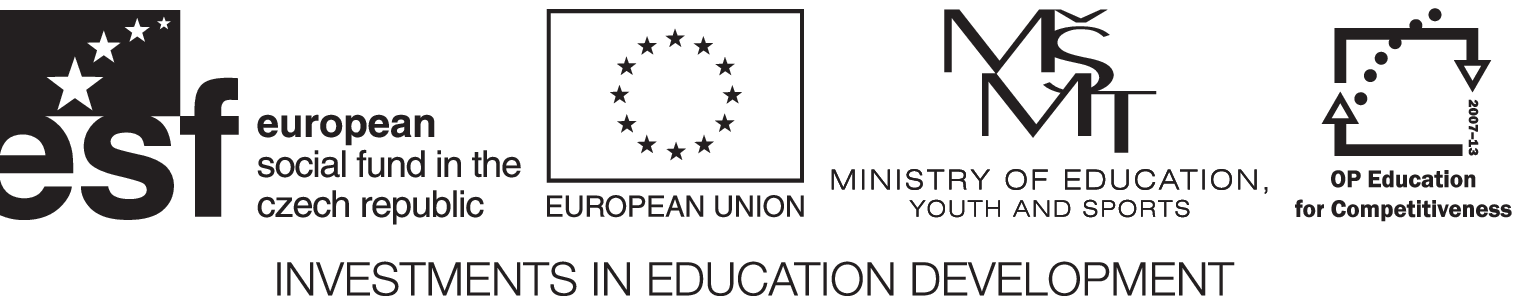}}
\end{tabular}}

\title[Galois connections and tense operators on q-effect algebras]%
{\bf Galois connections and tense operators\\  on q-effect algebras}
\author {Ivan Chajda, Jan Paseka}
\thanks{Both authors gratefully acknowledge  the support by ESF Project CZ.1.07/2.3.00/20.0051
Algebraic methods in Quantum Logic of the Masaryk University. \\
\logoesf\\%
I. Chajda acknowledges the support by a bilateral project I 1923-N25 
New Perspectives on Residuated Posets  financed by  
Austrian Science Fund (FWF) 
and the Czech Science Foundation (GA\v CR). 
J. Paseka gratefully acknowledges Financial Support 
of the Czech Science Foundation (GA\v CR) un\-der the grant 
Algebraic, many-valued and quantum structures for uncertainty modelling 
No.~GA\v CR  15-15286S}
\address{Palack\' y University Olomouc, Faculty of Sciences, t\v r. 17.listopadu 1192/12, Olomouc 771 46, Czech Republic}
\email{ivan.chajda@upol.cz}
\address{Department of Mathematics and Statistics, Faculty of Science,
Masaryk University, {Kotl\'a\v r{}sk\' a\ 2}, 611~37 Brno, 
Czech Republic}

\email{paseka@math.muni.cz}

\begin{document}

\begin{abstract}
For effect algebras, the so-called tense operators were already introduced 
by Chajda and Paseka. They presented also a canonical construction of 
them using the notion of a time frame.

Tense operators express the quantifiers
  ``it is always going to be the case that" and ``it has always been the case that" 
and hence enable us to express the dimension of time both 
in the logic of quantum mechanics and in the many-valued logic.

  A crucial problem concerning tense operators is their representation. Having an effect algebra with tense operators, 
we can ask if there exists a time frame such that each 
of these operators can be obtained by the canonical  construction. 
To approximate physical real systems as best as possible, 
we introduce the notion of a q-effect algebra and we solve this problem for q-tense operators on 
q-representable q-Jauch-Piron q-effect algebras.
\end{abstract}

\maketitle

\newtheorem{definition}{Definition}
\newtheorem{proposition}{Proposition}
\newtheorem{theorem}{Theorem}
\newtheorem{lemma}{Lemma}
\newtheorem{claim}{Claim}
\newtheorem{cor}{Corollary}
\newtheorem{corollary}{Corollary}
\newtheorem*{example}{Example}
\newtheorem{remark}{Remark}
\newtheorem{open}{Open problem}

\section{Introduction}

\label{intro}

{E}{ffect algebras were introduced by Foulis and Bennett \cite{FoBe} as} 
an abstraction of the Hilbert space effects
which play an important role in the logic of quantum mechanics.  However, this notion does not incorporate the
dimension of time. In fact, if an effect algebra $E$ is associated with a (possibly quantum-mechanical) physical system 
${\mathcal Q}$ under study, i.e. elements of $E$ represent 
propositions about the system ${\mathcal Q}$, the dimension 
of time can be expressed in terms of a one parameter 
group of continuous affine automorphisms on the state space 
$S$ of $E$. This group is usually called the dynamic group.

This means that effect algebras can serve to describe the states of effects in a given time but they
cannot reveal what these effects expressed in the past or what they will reveal in the next time. 
A similar problem was already solved for the 
classical propositional calculus by introducing the so-called tense operators 
$G$ and $H$ in Boolean algebras, see
\cite{1}. For MV-algebras and for {\L}ukasiewicz-Moisil algebras, tense operators 
were introduced by  Diaconescu and Georgescu in \cite{2}.   Contrary to Boolean algebras where the representation problem {through a time frame} 
is solved completely, authors in \cite{2} only mention that this problem for MV-algebras was not treated. 
Botur and Paseka were able to 
find a suitable time frame for given tense operators on a semisimple MV-algebra (see \cite{boturtense}), i.e., to 
solve the representation problem for semisimple MV-algebras.

Tense operators for effect algebras were 
introduced in \cite{dyn} by Chajda and Paseka and the corresponding tense logic was treated in 
\cite{chajdakolarik,dyn,dyn2} by Chajda, Janda, Kola\v{r}\'{i}k and Paseka. 

The concept of tense operators is closely connected 
with the concept of a time frame. It is a couple $(T,R)$ when $T$ is a time scale (both discrete or continuous) and 
$R$ is a relation of time preference, i.e., for $t_1, t_2\in T$
the expression $t_1Rt_2$ means ``$t_1$ is before $t_2$" and ``$t_2$ is after $t_1$", reflecting 
the fact that $R$ need not be an ordering. 

Having a time frame, we can derive tense operators $G, H, P$ and $F$ by a standard construction 
using suprema and infima. However, these tense operators can be given axiomatically   and this rises a natural question on the existence 
of a time frame for which these operators can be derived as mentioned above. This is called 
a representation problem. It is well known \cite{1} that this problem is easily solvable for Boolean algebras 
but it is not solvable e.g. for MV-algebras or effect algebras in general (see \cite{boturtense} and 
\cite{chajapas} for details). Hence, we are encouraged to find certain restrictions that 
are still in accordance with physical reality. These restrictions will enable us to solve the restricted problem 
for those algebras that are useful for axiomatization of many-valued and/or quantum logics.

The first restriction is to use the so-called Jauch-Piron states which correspond to behaviour of physical 
systems in quantum mechanics. The second one is to consider the so-called q-effect algebras instead 
of effect algebras which are equipped with two unary operations which, however, are term operations 
in the case of MV-algebras or lattice effect algebras. 

It is a rather surprising fact that these quite natural restrictions enable us to develop and use 
a completely different machinery based on using of Galois connections. The deeply developed 
theory of Galois connection can support an interesting construction of the preference 
relation $R$ on a time scale $T$ which is formed by means of Jauch-Piron q-states. 
This forms a time frame for given tense operators provided the set of these states is order reflecting.

It is worth noticing that the operators $G$ and $H$ can be considered as a 
certain kind of modal  operators which were already studied 
for intuitionistic calculus by Wijesekera 
\cite{wijesekera}, Chajda \cite{chajda} and in a  general setting by Ewald \cite{Ewald}, 
Chajda and Paseka \cite{dynmorpos}.

The paper is organized as follows. After introducing 
several necessary algebraic concepts, we introduce Galois q-connections and q-tense operators in a 
q-effect algebra, i.e., in an effect algebra equipped with two new unary operations $q$  and $d$.
 
In Section \ref {semistate} we introduce a powerful notion of a  q-semi-state on 
a q-effect algebra and we also establish a 
relationship among various kinds of q-semi-states on q-effect algebras. 
Following \cite{chajapas} we introduce the notion of a Jauch-Piron q-semi-state. 
The advantage of this approach is that these semi-states reflect an important property of the logic of quantum mechanics, 
namely the so-called Jauch-Piron property (see \cite{ptakpulm} and \cite{riecanova2}) saying that 
if the probability of propositions $A$ and $B$ being true is zero then there 
exists a proposition $C$ covering both $A$ and $B$.  Therefore we may expect applications 
of our method also in the realm of quantum logic.
We show that every Jauch-Piron q-semi-state can be created as an infimum of q-states.

In Section \ref{construction} we present a canonical construction of Galois 
q-connections and q-tense operators.   In Section \ref{repres}
we outline the problem of a representation of q-tense operators and we 
solve it for q-representable q-Jauch-Piron q-effect algebras.
This means that we get a procedure how to construct a corresponding time 
frame to be in accordance with the canonical construction from \cite{dyn}. 
Since any MV-algebra and any 
lattice effect algebra  with q-tense operators  is a q-effect algebra 
our approach covers the results from  \cite{boturtense} and \cite{chajapas}.

\section{{Preliminaries} and basic facts}
\label{Preliminaries}

By an {\em effect algebra} it is meant a structure $\mathcal{E}=(E;+,0,1)$ 
where $0$ and $1$ are distinguished elements of $E$, $0\neq 1$,
and $+$ is a partial binary operation on $E$ satisfying the following axioms for $x,y,z\in E$:
\begin{itemize}
  \item[(E1)] if $x+y$ is defined then $y+x$ is defined and $x+y=y+x$
  \item[(E2)] if $y+z$ is defined and $x+(y+z)$ is defined then $x+y$ and $(x+y)+z$ are defined and 
$(x+y)+z=x+(y+z)$
  \item[(E3)] for each $x\in E$ there exists a unique $x'\in E$ such that $x+x'=1$; $x'$ is called a {\em supplement} of $x$
  \item[(E4)] if $x+1$ is defined then $x=0$.
\end{itemize}
Having an effect algebra $\mathcal{E}=(E;+,0,1)$, we can introduce the {\em induced order} $\leq$ on $E$ 
and the partial operation $-$ as follows
  $$\begin{array}{l}
x\leq y\quad\text{ if for some }z\in E\quad x+z=y, \\
\text{and in this case} \ z=y-x
\end{array}$$
(see e.g. \cite{dvurec} for details). Then $(E;\leq )$ is an ordered set and $0\leq x\leq 1$ for each $x\in E$.

It is worth noticing that $a+b$ exists in an effect algebra $\mathcal{E}$ if and only if $a\leq b'$ (or equivalently, $b\leq a'$).
This condition is usually expressed by the notation $a\bot b$ (we say that $a,b$ are orthogonal).  Dually, we have a partial operation 
$\cdot$ on $E$ such that $a\cdot b$ exists in an effect algebra $\mathcal{E}$ if and only if $a'\leq b$ in which case $a\cdot b=(a'+b')'$. This allows us to equip 
$E$ with a dual effect algebraic operation such that 
$\mathcal{E}^{op}=(E;\cdot,1,0)$ is again an  effect algebra, 
$'^{\mathcal{E}^{op}}='^{\mathcal{E}}='$ and 
$\leq_{\mathcal{E}^{op}}=\leq^{op}$.

Let $\mathcal{E}=(E;+,0,1)$ be an effect algebra and $d, q:E\to E$ be maps such that, for all $x,y, z\in E$,  
\begin{itemize}
\item[(Q1)] $d(x')=q(x)'$,
\item[(Q2)] $d(0)=0=q(0), $
\item[(Q3)] $d$ is order-preserving, 
\item[(Q4)] $x'\leq x$ implies $x\cdot x=d(x)$, 
\item[(Q5)] $z\leq x, z\leq y$ and $y'\leq x$ imply $d(z)\leq x\cdot y$,
\end{itemize}
We then say that $\mathcal{E}=(E;+,q, d, 0,1)$  is a {\em q-effect algebra}.
Note that a dual of  $\mathcal{E}=(E;+,q, d, 0,1)$ is a  q-effect algebra 
$\mathcal{E}^{op}=(E;\cdot, d, q, 1,0)$.

A {\em morphism  of effect algebras} ({\em morphism  of q-effect algebras}) is a map between
them such that it preserves the partial 
operation $+$, (and the unary operations $q$ and $d), $the bottom and the top elements. In particular, 
$':{\mathcal{E}}\to {\mathcal{E}^{op}}$ is a morphism  of effect algebras (morphism  of q-effect algebras).

A  {\em morphism} $f:P_1\to P_2$ {\em of   posets} 
is an order-preserving map. 
Any morphism  of effect algebras is a morphism of 
corresponding bounded  posets. 
A morphism $f:P_1\to P_2$ of bounded  posets 
is {\em order reflecting} if ($f(a)\leq f(b)$ if and only if $a\leq b$)  
for all $a, b\in P_1$. 

\begin{observation} \label{obsik} Let $P_1, P_2$ be 
bounded posets, $T$ a set and 
$h_{t}:P_1\to P_2, t\in T$ morphisms of bounded posets.
The following conditions are equivalent:
\begin{enumerate}
\item[{\rm(i)}] \(((\forall t \in T)\, h_{t}(a)\leq h_{t}(b))\implies a\leq
b\) for any elements \(a,b\in P_1\);
\item[{\rm(ii)}] The map $h:P_1 \to P_2^{T}$ defined by 
$h(a)=(h_t(a))_{t\in T}$ for all $a\in P_1$ is order reflecting.
\end{enumerate}
\end{observation}

We then say that $\{h_{t}:P_1\to P_2; t\in T\}$ is an {\em order reflecting set of 
order preserving maps with respect to} $P_2$. Note that we may in this case 
identify $P_1$ with a subposet of $P_2^{T}$ since $h$ is an injective 
morphism of bounded posets.

If moreover $(E;\leq)$ is a lattice (with respect to the induced order), then  $\mathcal{E}$ is called a 
 {\em lattice effect algebra}.  On any lattice effect algebra  $\mathcal{E}$ 
we may introduce total binary operations $\oplus$ and $\odot$ as follows: 
$x\oplus y=x+(y\wedge x')$ and $x\odot y= (x'\oplus y')'$. 
In this case the unary operations 
$q(x)=x \oplus x$ and $d(x)=x \odot x$ satisfy 
the conditions (Q1)-(Q5) and $\mathcal{E}=(E;+,q, d, 0,1)$ is a  q-effect algebra.
Note that 
a lattice effect algebra 
$\mathcal{E}$ is an MV-algebra (see \cite{Mun} for more details on MV-algebras) 
with respect to the operations $\oplus$ and $'$  if and 
only if $x\wedge y=0$ implies $x\leq y'$.  
In particular, any linearly ordered  q-effect algebra is an MV-algebra. Moreover, 
any morphism of MV-algebras is a morphism of  q-effect algebras.

In what follows we will present a genuine example of a q-effect algebra 
which is not a lattice effect algebra.

\begin{example}\label{qeanotlea} \upshape Let 
$E=\{{\mathbf 0}=(0,0), a=(\frac{5}{6},0), b=(\frac{1}{6},\frac{1}{6}), %
c=(0,\frac{5}{6}), a+b=({1},\frac{1}{6}),  2b=(\frac{2}{6},\frac{2}{6}), %
3b=(\frac{3}{6},\frac{3}{6}), 4b=(\frac{4}{6},\frac{4}{6}), %
5b=a+c=(\frac{5}{6},\frac{5}{6}), b+c=(\frac{1}{6},1), 
{\mathbf 1}=6b=a+b+c=({1},{1})\}\subseteq [0,1]^{2}$. Then 
$E$ is an effect algebra as described in \cite[Example 2.3]{pasnied} and depictured in 
Fig. \ref{Fig2v}. 

 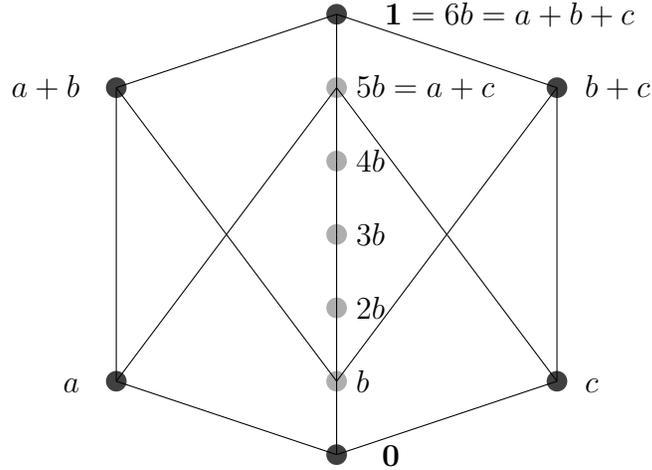
\begin{figure}[h]
\centering
\begin{tikzpicture}[scale=0.976583]
\coordinate [label=right:\phantom{ll}\hbox{$\phantom{ll}{\mathbf 1}=6b=a+b+c$}] (1) at (0,6);
\coordinate [label=right:\phantom{ll}\hbox{$c$}\phantom{lll}] (c) at (3,1);
\coordinate [label=left:\phantom{lll}\hbox{$a$}\phantom{lll}] (a) at (-3,1);
\coordinate [label=left:\phantom{lll}\hbox{$a+b$}\phantom{lll}] (a+b) at (-3,5);
\coordinate [label=right:\phantom{ll}\hbox{$b+c$}\phantom{lll}] (b+c) at (3,5);
\coordinate [label=right:\phantom{l}\hbox{$b$}\phantom{lll}] (b) at (0,1);
\coordinate [label=right:\phantom{l}\hbox{$2b$}\phantom{lll}] (2b) at (0,2);
\coordinate [label=right:\phantom{l}\hbox{$3b$}\phantom{lll}] (3b) at (0,3);
\coordinate [label=right:\phantom{l}\hbox{$4b$}\phantom{lll}] (4b) at (0,4);
\coordinate [label=right:\phantom{l}\hbox{$5b=a+c$}\phantom{lll}] (5b) at (0,5);
\coordinate [label=right:\phantom{llll}\hbox{${\mathbf 0}$}] (0) at (0,0);
\draw (0) -- (a); 
\draw (0) -- (c);
\draw (0) -- (b) ; 
\draw (5b) -- (1);
\draw (a+b) -- (1);
\draw (b+c)  --  (1);
\draw (a)  --  (5b);
\draw (c)  --  (5b);
\draw (a)  --  (a+b);
\draw (c)  --  (b+c);
\draw (b)  --  (2b);
\draw (b)  --  (a+b);
\draw (b)  --  (b+c);
\draw (2b)  --  (3b);
\draw (3b)  --  (4b);
\draw (4b)  --  (5b);
\fill[black,opacity=.75] (0) circle (4pt);
\fill[black,opacity=.75] (1) circle (4pt);
\fill[black,opacity=.75] (a) circle (4pt);
\fill[black,opacity=.32075] (b) circle (4pt);
\fill[black,opacity=.32075] (2b) circle (4pt);
\fill[black,opacity=.32075] (3b) circle (4pt);
\fill[black,opacity=.32075] (4b) circle (4pt);
\fill[black,opacity=.32075] (5b) circle (4pt);
\fill[black,opacity=.75] (c) circle (4pt);
\fill[black,opacity=.75] (b+c) circle (4pt);
\fill[black,opacity=.75] (a+b) circle (4pt);
\end{tikzpicture}
\caption{Figure of underlying poset of a q-effect  algebra which is not lattice ordered}\label{Fig2v}
\end{figure}  

The unary operations $q$ and $d$ on $E$ will be defined as follows:
$$
\begin{array}{|c|c|c|c|c|c|c|c|c|c|c|c|}
\hline
\rule{0mm}{0.45cm}x&{\mathbf 0}&a&b&c&a+b&2b&3b&4b&5b&b+c&{\mathbf 1}\\ \hline
\rule{0mm}{0.45cm}q(x)&{\mathbf 0}&a&2b&c&a+b&4b&{\mathbf 1}&{\mathbf 1}&{\mathbf 1}&b+c&{\mathbf 1}\\ \hline
\rule{0mm}{0.45cm}d(x)&{\mathbf 0}&a&{\mathbf 0}&c&a+b&{\mathbf 0}&{\mathbf 0}&2b&4b&b+c&{\mathbf 1}\\ \hline
\end{array}\phantom{i}.
$$
Clearly, the conditions (Q1)-(Q5) are satisfied. Therefore, 
$(E;+,q, d, {\mathbf 0},{\mathbf 1})$ is a q-effect algebra. Note that 
the join of $a$ and $b$ does not exist in $E$, i.e., 
$(E;+,{\mathbf 0},{\mathbf 1})$ is not lattice ordered.
\end{example}

A {\em standard q-effect algebra $\mathcal I$} is a structure 
$\mathcal I=([0,1];+,q, d, 0,1)$ where $[0,1]$ is the real unit interval, 
$x + y$ is defined iff the usual sum of $x$ and $y$ is less or equal to  $1$ in which case 
$x+y$ coincides with it, $q(x)=x+\min(1-x, x)=x\oplus x$, $d(x)=x-\min(1-x, x)=x\odot x$. 
The corresponding {\em standard MV-algebra} will be simply denoted as $[0,1]$.

A map $s:E\to[0,1]$ is called a
{\em state} (a {\em q-state}) on $\mathcal{E}$ if $s(0)=0$, $s(1)=1$, 
($s(q(x))=s(x)+\min(s(x'), s(x))=s(x)\oplus s(x)$, 
$s(d(x))=1-s(x')-\min(s(x'), s(x))=s(x)\odot s(x)$) and 
$s(x+ y)=s(x)+s(y)$ whenever $x+ y$ exists in $\mathcal{E}$. 
In particular, a q-state $s$ is morphism of  q-effect algebras $s:\mathcal{E}\to \mathcal{I}$.


\subsection*{Riesz decomposition property, filters and ideals in effect algebras} 

We recall that an effect algebra  $\mathcal{E}$ satisfies the {\em Riesz
Decomposition Property} ((RDP) in abbreviation) if  $x \le y_1 +
y_2$ implies that there exist two elements $x_1, x_2 \in E$ with
$x_1 \le y_1 $ and $x_2 \le y_2$ such that $x = x_1 + x_2$.

An {\em ideal} of an effect algebra $\mathcal{E}$ is a non-empty subset $I$ of
$E$ such that 
\begin{enumerate}
\item[(i)] $x \in E$, $y \in I$, $x\le y$ imply $x \in I$,
\item[(ii)] if $x,y \in I$ and $x+y$ is defined in $\mathcal{E}$, then $x+y \in
I$.  
\end{enumerate}

A {\em filter} of an effect algebra $\mathcal{E}$ is an ideal in the dual effect algebra $\mathcal{E}^{op}=(E;\cdot,1,0)$.

We denote by ${Id}(\mathcal{E})$ the set of all ideals of $\mathcal{E}$. 
An ideal $I$ is said to be a {\em Riesz ideal} if, for $x \in I$,
$a,b \in E$ and $x \le a+b$, there exist $a_1,b_1 \in I$ such that
$x = a_1 +b_1$ and $a_1 \le a$ and $b_1 \le b$.

For example, if $\mathcal{E}$ has   (RDP), then any ideal $I$  
of $\mathcal{E}$ is Riesz (see \cite{Dvu4}) and, moreover, 
we obtain a congruence   $\sim_I$ on $\mathcal{E}$ 
which is given by $a\sim_I b$ iff there are $x,y \in I$ with $x\leq a$ and $y\leq b$ such 
that $a-x = b-y$.  It follows that  the quotient  $\mathcal{E}/I$ is an effect algebra with RDP.  
Recall that  if $\mathcal{E}$ carries a structure of an MV-algebra then the notions of an ideal and an MV-algebra ideal coincide. 
In this case, the factor MV-morphism is also a morphism of q-effect algebras.

\subsection*{Dyadic numbers and MV-terms}
\label{Dyadic}

The content of this part summarizes the basic results about certain MV-terms 
from \cite{teheux} in the setting of q-effect algebras.

The set $\mathbb D$ of dyadic 
numbers is the set of the rational numbers that can be written as a finite sum of powers of 2.
We denote 
by $T_{\mathbb D}$  the clone generated by $q$ and $d$.
Let us define, for any $m\in \mathbb{N}$ a term $\mu_{m}\in T_{\mathbb D}$. 
First, if $m=1$ we put  $\mu_{1}=d$. Second, assume that 
$\mu_{m}$ is defined. Then we put $\mu_{m+1}=\mu_1\circ \mu_{m}$.

We will need the following:

\begin{corollary}{\rm \cite[Corollary 1.15 (1)]{teheux}} \label{odhad} Let $\mathcal I$ be the 
standard q-effect algebra, $x\in [0, 1]$ and $r\in (0,1)\cap {\mathbb D}$. Then 
 there is a term $t_r$ in $T_{\mathbb D}$ such that
$$
t_r(x)=1 \quad \text{if and only if}\quad r\leq x.
$$
\end{corollary}

\begin{proposition}{\rm \cite[Proposition 1]{boturtense}}\label{T4} Let $\mathcal{E}$ be a linearly ordered 
q-effect algebra, $s:E\to [0,1]$ a q-state on $\mathcal{E}$, $x\in E$.  Then 
$s(x)=1$ iff $t_r(x)=1$ for all $r\in (0,1)\cap {\mathbb D}$. 

Equivalently, $s(x)<1$ iff there is a dyadic number $r\in (0,1)\cap {\mathbb D}$ 
such that $t_r(x)\not=1$. In this case, $s(x)<r$. 
\end{proposition}

\begin{corollary}\label{obind}  Let $\mathcal{E}=(E;+,q, d, 0,1)$  be a q-effect algebra, 
$k\in \mathbb{N}, h_1, \dots, h_{2^{k}}, h\in E$ such that 
$h\leq h_j, 1\leq j\leq {2^{k}}$ and $ h_1{}\cdot{}  \dots  {}\cdot{} h_{2^{k}}$ exists. Then 
$\mu_{k}(h)\leq   h_1{}\cdot{}  \dots  {}\cdot{} h_{2^{k}}$.
\end{corollary}
\begin{proof}
It follows by an obvious induction with respect to $k$. 
\end{proof}

\subsection*{Galois connections}

Let $(A;\leq)$ and $(B;\leq)$ be ordered sets. A mapping $f:A\to B$ is 
called {\em residuated} if there exists a mappping $g:B\to A$ 
such that 
$$f(a)\leq b\quad\text{if and only if}\quad a\leq g(b)$$
for all $a\in A$ and $b\in B$. 

In this situation, we say that $f$  and $g$ form a {\em residuated pair} or that 
the pair $(f,g)$ is a (monotone) {\em Galois connection}. 
The mapping $f$ is called a {\em lower adjoint of $g$} or a
{\em left adjoint of $g$}, the mapping $g$ is called an 
{\em upper adjoint of $f$} or a 
{\em right adjoint of $f$}. 

Galois connections can be described as follows.

\begin{lemma}\label{GalCon}
Let $(A;\leq)$ and $(B;\leq)$ be ordered sets. Let  $f:A\to B$  and $g:B\to A$ be mappings. The following conditions are equivalent:
\begin{enumerate}
\item $(f,g)$ is a {Galois connection}.
\item $f$ and $g$ are monotone, $\mathop{id}_A\leq g\circ f$ and $ f\circ g\leq \mathop{id}_B$. 
\item $g(b)=\sup\{x\in A \mid f(x)\leq b\}$ and $f(a)=\inf\{y\in B\mid a\leq g(y)\}$ for all $a\in A$ and $b\in B$.
\end{enumerate}
\end{lemma}

In the above case, $g$ is determined uniquely by $f$ and, similarly, $f$ is determined uniquely by $g$.  
Moreover, $f$ preserves all existing joins in $(A;\leq)$ and 
 $g$ preserves all existing meets in $(B;\leq)$. If, in addition, both $(A;\leq)$ and $(B;\leq)$ are complete 
ordered sets we have the converse, i.e. if $f$ preserves all  joins in $(A;\leq)$ then 
$f$ has an upper adjoint $g$ given by the condition $g(b)=\sup\{x\in A \mid f(x)\leq b\}$, 
for all $b\in B$. Similarly,  if $g$ preserves all  meets in $(A;\leq)$ then 
$g$ has a lower adjoint $f$ given by the condition $f(a)=\inf\{y\in B\mid a\leq g(y)\}$,  
for all $a\in A$.

We now present a most prominent example of monotone Galois connection related to temporal logic.

\begin{example} {\em Let $A$ and $B$ are arbitrary sets and let $R\subseteq A\times B$. 
For every $X\subseteq A$ and every $Y\subseteq B$, we define 
$$f_R(X)=\{b\in B \mid \exists x\in X \ \text{such that}\ x R b\}$$
and 
$$g_R(Y)=\{a\in A \mid \left(\forall y\in B\right) \left( a R y\ \text{implies}\ y\in Y\right)\}.$$
Then 
$f_R(X)\subseteq Y$ $\ekviv$ $\forall x\in A, \forall y\in B \left(x\in X\ \text{and}\ xRy\ \text{implies}\ y\in Y\right)$ 
 $\ekviv$  $X\subseteq g_R(Y)$. Hence the pair $(f_R, g_R)$ forms a Galois connection 
between power sets $(\wp(A);\subseteq)$ and $(\wp(B);\subseteq)$ and any Galois connection 
between $(\wp(A);\subseteq)$ and $(\wp(B);\subseteq)$ is of this form.

}
\end{example}

\subsection*{Galois connections and tense operators on q-effect algebras}\hfill

\medskip


Let $\mathcal{E}_1=(E_1;+_1,q_1, d_1, 0_1,1_1)$ and 
$\mathcal{E}_2=(E_2;+_2,q_2, d_2,0_2,1_2)$  be  q-effect algebras, 
$f:E_1\to E_2$ and $g:E_2\to E_1$ be a mappings such that 
$(f,g)$ is a Galois connection. 
We say that $(f,g)$ is a {\em Galois q-connection} if 

 \begin{itemize}
    \item[(GQ1)] \begin{tabular}{@{}l}$f(q_1(x))=q_2(f(x))$,\\ 
    $f(d_1(x))=d_2(f(x))$,
    \end{tabular}
\item[] 
    \item[(GQ2)] \begin{tabular}{@{}l}$g(q_2(y))=q_1(g(y))$,\\ 
     $g(d_2(y))=d_1(g(y))$
    \end{tabular} 
 \end{itemize}
for all $x\in E_1$ and $y\in E_2$.

Note that then the mappings $\overline{g}={\,} '^{{}_{1}}\circ{\,} g\circ{\,} '^{{}_{2}}:E_2\to E_1$ 
and $\overline{f}={\,} '^{{}_{2}}\circ{\,} f\circ{\,} '^{{}_{1}}:E_1\to E_2$ form 
 a {Galois q-connection} $(\overline{g}, \overline{f})$.

We then have the following. 

\begin{lemma}\label{RgRf}
Let $\mathcal{E}_1=(E_1;+_1,q_1, d_1, 0_1,1_1)$, 
$\mathcal{E}_2=(E_2;+_2,q_2, d_2,0_2,1_2)$ and 
$\mathcal{E}_3=(E_3;+_3,$ $q_3, d_3,0_3,1_3)$  be  q-effect algebras, 
$f:E_1\to E_2$ and $g:E_2\to E_1$ be mappings such that 
$(f,g)$ is a Galois q-connection, $s:E_1\to E_3$ and $t:E_2\to E_3$ order-preserving mappings. 
Then 
$$
(s\circ g)(x)\leq t(x)\ \text{for all}\ x\in E_2\ \text{if and only if }\ 
s(y)\leq (t\circ f)(y)\ \text{for all}\ y\in E_1.
$$
If, moreover,  $s$ and $t$ are morphisms of effect algebras we have 
that 
$$
(s\circ g)(x)\leq t(x)\ \text{for all}\ x\in E_2\ \text{if and only if }\ 
t(z)\leq (s\circ \overline{g})(z)\ \text{for all}\ z\in E_2
$$
and 
$$
s(y)\leq (t\circ f)(y)\ \text{for all}\ y\in E_1 \ \text{if and only if }\  
 (t\circ \overline{f})(w)\leq s(w)\ \text{for all}\ w\in E_1.
$$
\end{lemma}
\begin{proof} Assume first that $(s\circ g)(x)\leq t(x)\ \text{for all}\ x\in E_2$. Let 
$y\in E_1$. Then $s(y)\leq s(g(f(y))\leq t(f(y))$. Conversely, let 
$s(y)\leq (t\circ f)(y)\ \text{for all}\ y\in E_1$. Let $x\in E_2$. 
Then $t(x)\geq t(f(g(x))\geq s(g(x))$. 

Assume moreover that $s:E_1\to E_3$ and $t:E_2\to E_3$ are morphisms of effect algebras and 
that $(s\circ g)(x)\leq t(x)\ \text{for all}\ x\in E_2$ holds. Then, 
$(s(g(x))'\geq t(x)'\ \text{for all}\ x\in E_2$, i.e., 
$(s(g(x)')\geq t(x')\ \text{for all}\ x\in E_2$. Put $z=x'$. Then 
$(s(\overline{g}(z))\geq t(z)\ \text{for all}\ z\in E_2$. The remaining parts follow by same 
considerations.
\end{proof}

Let $\mathcal{E}=(E;+,q, d, 0,1)$   be a  q-effect algebra.
Unary operators $G$ and $H$ on 
$\mathcal{E}$ are called {\em  q-tense operators}  if the mappings 
$P= {\,}'\circ{\,} H\circ{\,} '$ and $G$ form a  Galois q-connection. Then 
also the mappings 
$F={\,} '\circ{\,} G\circ{\,} '$ and $H$ form a  Galois q-connection.
In particular, $G$ and $H$ 
satisfy  the following conditions:
  \begin{itemize}
    \item[(T1)] $G(1)=H(1)=1$,
    \item[(T2)] $x\leq y$ implies $G(x)\leq G(y)$  and $H(x)\leq H(y)$,
    \item[(T3)] $x\leq GP(x)$ and  $x\leq HF(x)$, 
 \item[(T4)] \begin{tabular}{@{}l}$G(q(x))=q(G(x))$,\\ 
    $H(q(x))=q(H(x))$,
    \end{tabular}
    \item[(T5)] \begin{tabular}{@{}l}$G(d(x))=d(G(x))$,\\ 
     $H(d(x))=d(H(x))$.
    \end{tabular}  
  \end{itemize}
for all $x, y\in E$.

It is quite natural require  our  tense operators on q-effect 
algebras to preserve unary operations $q$ and $d$ (see \cite{2}). The main aim of our paper 
is to establish a representation theorem for q-tense operators.

\begin{lemma} Let $f:{E}_1\to {E}_2$  be a mapping between q-effect algebras  $\mathcal{E}_1=(E_1;+_1,q_1, d_1,$ $0_1,1_1)$ and 
$\mathcal{E}_2=(E_2;+_2,q_2, d_2,0_2,1_2)$ satisfying the condition (GQ1),  
$r\in (0,1)\cap {\mathbb D}$. Then 
$t_r(f(x))= f(t_r(x))$ for all $x\in E_1$. 
\end{lemma}
\begin{proof} It follows from the fact that  $t_r\in T_{\mathbb D}$ is defined inductively using only 
the operations $q$ and $d$. 
\end{proof}

\section{\mbox{$\mathrm q$}-semi-states on \mbox{$\mathrm q$}-effect algebras} \label{semistate}

The aim of this section is to establish a 
relationship among various kinds of q-semi-states on q-effect algebras and to 
show that every Jauch-Piron q-semi-state can be created as an infimum of 
q-states.

\begin{definition}\label{semist}{\rm  Let $\mathcal{E}=(E;+,q, d, 0,1)$  be a q-effect algebra and 
let $\mathcal I$ be the standard q-effect algebra. A map 
$s:{E} \to [0,1]$ is called
\begin{enumerate}
\item {\em a q-semi-state on}  $\mathcal{E}$ if 
\begin{itemize}
\item[(i)] $s(1)=1,$
\item[(ii)] $s(x)\leq s (y)$ whenever $x\leq y$,
 \item[(iii)] $s(x)\odot s(x)=s(d(x)),$ 
\item[(iv)] $s(x)\oplus s(x)=s(q(x)),$
\end{itemize}

\item {\em a Jauch-Piron q-semi-state on}  $\mathcal{E}$ if $s$ is  a q-semi-state and 
\begin{itemize}
\item[(v)] $s(x)=1= s(y)$ implies  $s(z)=1$ for some $z\in E$, $z\leq x$ and $z\leq y$;
\end{itemize}
\item {\em a strong q-semi-state on}  $\mathcal{E}$ if $s$ is  a q-semi-state and 
\begin{itemize}
\item[(vi)] $s(x)=1= s(y)$ and $x\cdot y$   defined implies  $s (x\cdot y)=1$.
\end{itemize}
\end{enumerate}
}
\end{definition}

Note that any q-state on  $\mathcal{E}$ is a q-semi-state.

Moreover, if $\mathcal{E}$ is an MV-algebra  and $s$ a state on $\mathcal{E}$ (in the effect algebraic sense, see also \cite{Mun}) 
then we always have 
$s(x\vee y)+s(x\wedge y)=s(x)+s(y)$  for all $x, y\in E$. Hence for every 
MV-algebra any its state $s$ satisfies (v).

\begin{lemma}\label{JPmeetstate} Let $\mathcal{E}=(E;+,q, d, 0,1)$  be a q-effect algebra, 
$s:{E} \to [0,1]$ a Jauch-Piron q-semi-state on  $\mathcal{E}$. Then 
$s$ is a strong q-semi-state.
\end{lemma}
\begin{proof} Assume that $s(x)=1= s(y)$ and $x\cdot y$ is  defined. Then there is an 
element  $z\in E$ such that $z\leq x, z\leq y, y'\leq x$ and  $s(z)=1$. Therefore also 
$d(z)\leq x\cdot y$  which yields  
$
\begin{array}{@{}r@{\,}c@{\,}l}
1&=&s(z)=s(z)\odot s(z)=s(d(z))\leq s(x\cdot y).
\end{array}
$\end{proof}

\begin{lemma}\label{meetstate} Let $\mathcal{E}=(E;+,q, d, 0,1)$  be a q-effect algebra, 
$S$ a non-empty set of q-semi-states on $\mathcal{E}$. Then
\begin{enumerate}
\item[(a)] the pointwise meet 
$t=\bigwedge S:\mathcal{E} \to [0,1]$ is a q-semi-state on  $\mathcal{E}$,
\item[(b)] if $S$ is linearly ordered then 
$q=\bigvee S:\mathcal{E} \to [0,1]$ is a q-semi-state on  $\mathcal{E}$.
\end{enumerate}
\end{lemma}
\begin{proof}
The proof is a straightforward checking of conditions (i)-(iv). 
\end{proof}

\begin{definition}\label{JPrepsemist}{\rm  Let $\mathcal{E}=(E;+,q, d, 0,1)$  be a q-effect algebra.  
\begin{enumerate}
\item[(a)] 
If $S$ is an order reflecting set of q-states on  $\mathcal{E}$ then $\mathcal{E}$ is said to be {\em q-representable}.
\item[(b)] 
If $S$ is an order reflecting set of 
Jauch-Piron q-states on  $\mathcal{E}$ then $\mathcal{E}$ is said to be {\em q-Jauch-Piron representable}.
\item[(c)] If any q-state is q-Jauch-Piron then $\mathcal{E}$ is called an  {\em q-Jauch-Piron q-effect algebra}.
\end{enumerate}}
\end{definition}

First, note that any q-Jauch-Piron q-effect algebra with an order reflecting set of q-states is  q-Jauch-Piron representable. 
Also, any  q-Jauch-Piron representable q-effect algebra is  q-representable. 

Second, if $\mathcal{E}$ is {q-representable} then the induced morphism 
${i_{\mathcal{E}}^S}:{\mathcal{E}} \to [0,1]^{S}$ (sometimes called {\em an embedding}) 
is an order reflecting morphism 
of effect algebras such that ${i_{\mathcal{E}}^S}(q(x))={i_{\mathcal{E}}^S}(x)\oplus {i_{\mathcal{E}}^S}(x)$ 
and ${i_{\mathcal{E}}^S}(d(x))={i_{\mathcal{E}}^S}(x)\odot {i_{\mathcal{E}}^S}(x)$ 
for all $x\in E$. Here the operations $\oplus$ and $\odot$ are defined componentwise.

\begin{lemma}\label{porovb}  Let  $\mathcal{E}=(E;+,q, d, 0,1)$  be a q-effect algebra, $s, t$ 
q-semi-states on  $\mathcal{E}$. Then $t\leq s$ iff 
$t(x)=1$ implies $s(x)=1$ for all $x\in E$.
\end{lemma}
\begin{proof} Clearly,  $t\leq s$ yields the condition
$t(x)=1$ implies $s(x)=1$ for all $x\in E$.

Assume now that $t(x)=1$ implies $s(x)=1$ for all $x\in E$ is valid and 
that there is $y\in E$ such that $s(y) < t(y)$. Thus, there is a 
dyadic number $r\in (0,1)\cap\mathbb D$ such that $s(y)<r< t(y).$ 
By Corollary \ref{odhad} there is a term 
 $t_r$ in $T_{\mathbb D}$ such that
$t_r(s(y))< 1$ and $t_r(t(y))=1$. It follows that 
$s(t_r(y))=t_r(s(y))< 1$ and $t(t_r(y))=t_r(t(y))=1$. The last condition yields that 
$s(t_r(y))=1$, a contradiction.
\end{proof}

\begin{theorem}\label{prusekapr}  Let $\mathcal{E}=(E;+,q, d, 0,1)$  be a q-effect algebra 
with an order reflecting set $S=\{ s: E\to [0,1]\mid \ s\ \text{is a}\ \text{q-state on}$  $\mathcal{E}\}$, $t$ 
a  Jauch-Piron  q-semi-state on  $\mathcal{E}$ and 
$S_{t}=\{ s\in S \mid   s\geq t\}$. 
Then $t=\bigwedge S_t$.
\end{theorem}
\begin{proof} We may assume that $E\subseteq [0,1]^{S}$ such that $x+y$, $x\cdot y$, $q(x)$ and $d(x)$ computed in $\mathcal{E}$ gives us 
the same results as $x+y$, $x\cdot y$, $x\oplus x$ and $x\odot x$ computed in $[0, 1]^{S}$ and restricted to elements from $E$. This means that 
the inclusion map $i:\mathcal{E}\to [0, 1]^{S}$ is an order reflecting morphism of q-effect algebras.
Note also that $[0, 1]^{S}$ is a complete lattice q-effect algebra with RDP (MV-algebra).

Clearly,  $t\leq\bigwedge S_t$. Assume that there is 
$x\in E$ such that  $t(x)<\bigwedge S_t (x)$. Thus, there is a 
dyadic number $r\in (0,1)\cap\mathbb D$ such that $t(x)<r<\bigwedge S_t (x).$ 
Again by  Corollary \ref{odhad} there is a term 
 $t_r$ in $T_{\mathbb D}$ such that 
$t(t_r(x))=t_r(t(x))<1$.

Let us put 
$U=\{z\in E\mid t(z)=1\}$.  
Clearly, $z\in U$, $w\in E$ and $w\geq z$ implies $w\in U$. 
For all $f_1, \dots, f_n\in U$ there is an element $f\in U, f\leq f_i$ ,  $i=1, \dots, n$ and  
$z\in U$ yields $\mu_k(z)\in U$ for all $k\in\mathbb{N}$ since $t$ is a Jauch-Piron q-semi-state, 
$t_r(x)\not\in U$.

Let $V$ be a filter of $[0, 1]^{S}$ generated by the set $U$. 
Then  by \cite[Proposition 3.1]{Dvu4} 
$$
\begin{array}{r@{\,} l}
V=\{y\in [0, 1]^{S} \mid &\exists n\in\mathbb{N}, \exists g_1, \dots, g_n\in [0, 1]^{S}, \\
&\exists f_1, \dots, f_n\in U, f_i\leq g_i,\\ 
&i=1, \dots, n, y=g_1\cdot{} \dots {}\cdot g_n\}.
\end{array}
$$

Let us assume that $t_r(x)\in V$. Then  $\exists n\in\mathbb{N}, \exists g_1, \dots, g_n\in [0, 1]^{S}, 
\exists f_1, \dots, f_n\in U, f_i\leq g_i, i=1, \dots, n, t_r(x)=g_1\cdot{} \dots {}\cdot g_n$. There is 
an element $f\in U$ such that $ f\leq f_i$ ,  $i=1, \dots, n$. 
Let $k\in\mathbb{N}$ be minimal such that $n\leq 2^{k}$. It follows by Corollary \ref{obind} 
that $t_r(x)=g_1\cdot{} \dots {}\cdot g_n \cdot \underbrace{1 \cdot{} \dots {}\cdot 1}_{2^{k}-n\ \text{times}} 
\geq \mu_{k}(f)\in U$, i.e., $t_r(x)\in U$, a contradiction. 

So we have that $t_r(x)\notin V$. Let $W$ be a maximal filter of $[0, 1]^{S}$  which does contain $V$ and does not 
contain $t_r(x)$. Then the set $I=\{y\in [0, 1]^{S} \mid y'\in W\}$ is a prime ideal in $[0, 1]^{S}$. It follows 
by \cite[Proposition 6.5 and Proposition 6.10]{Dvu4} that 
$[0, 1]^{S}/I$ is a linearly ordered q-effect algebra, i.e., an MV-algebra such that $t_r([x]_{I})\not=[1]_I$. 
In particular the factor map $\pi_I:[0, 1]^{S}\to [0, 1]^{S}/I$ is  an MV-morphism such that $\pi_I(x)=[x]_{I}$ 
and $\pi_I(F)\subseteq \pi_I(W)=\{[1]_I\}$.

Let us denote by $U_I$ the maximal 
ideal of $ [0, 1]^{S}/I$ and by $\overline{s}: [0, 1]^{S}/I \to [0,1]$ the 
corresponding MV-morphism.
Hence by Proposition \ref{T4} $\overline{s}([x]_{I})<r<1$. Let us put $s=\overline{s}\circ \pi_I\circ i$. Then 
$s$ is a q-state such that 
$s(F)=\overline{s}(\pi_I(F))=\overline{s}(\{[1]_I\})=\{1\}$. It follows by Lemma \ref{porovb} 
that  $t\leq s$, i.e., $s\in S_t$ and  $s(x)=\overline{s}([x]_{I})< r <\bigwedge S_t (x) \leq  s(x)$, 
a contradiction.
\end{proof}

\begin{corollary}\label{nenul} Let $\mathcal{E}=(E;+,q, d, 0,1)$  be a q-effect algebra 
with an order reflecting set $S$ of q-states on  $\mathcal{E}$.  
The only  Jauch-Piron  q-semi-state $t$ on  $\mathcal{E}$
with $t(0)\not=0$ is the constant function $t(x)=1$ for all $x\in E$.
\end{corollary}
\begin{proof}
Clearly, the set $S_{t}=\{ s\in S \mid   s\geq t\}$ from Theorem \ref{prusekapr} has to 
be empty (otherwise we would have $t(0)=\bigwedge S_t(0)= \bigwedge \{ 0\}=0$, 
a contradiction) which yields that $t(x)=(\bigwedge S_t)(x)= \bigwedge \emptyset=1$.
\end{proof}

\begin{corollary}\label{nenul} Let $\mathcal{E}=(E;+,q, d, 0,1)$  be a q-effect algebra 
with an order reflecting set $S$ of q-states on  $\mathcal{E}$.  
Then all   Jauch-Piron  q-semi-states $t$ on  $\mathcal{E}$
with $t(0)=0$ satisfy the condition  
\begin{itemize}
\item[(ii')]$t(x)+ t(y)\leq t (x+ y)$ whenever $x+ y$ is defined.
\end{itemize}
\end{corollary}
\begin{proof}
Clearly, the set $S_{t}=\{ s\in S \mid   s\geq t\}$ from Theorem \ref{prusekapr} is 
non-empty, its elements satisfy (ii') and they map $0$ to $0$. 
As in \cite[Lemma 6]{chajapas} we get (ii').
\end{proof}

\begin{corollary}\label{dusprusekapr}  Let $\mathcal{E}=(E;+,q, d, 0,1)$  be a q-Jauch-Piron  
q-effect algebra with an order reflecting set 
$S$ of all  Jauch-Piron q-states on  $\mathcal{E}$, $t$ 
a  Jauch-Piron  q-semi-state on  $\mathcal{E}$ and 
$S_{t}=\{ s\in S \mid   s\geq t\}$. 
Then $t=\bigwedge S_t$.
\end{corollary}
\begin{proof}It follows immediately from Proposition \ref{prusekapr} since 
$$\begin{array}{r@{\,}c@{\,}l}S&=&\{ s: E\to [0,1]\mid \ s\ \text{is a}\ \text{q-state on}\ \mathcal{E}\}\\[0.2cm]
&=&\{ s: E\to [0,1]\mid \ s\ \text{is a}\ \text{Jauch-Piron q-state on}\ \mathcal{E}\}.
\end{array}$$
\end{proof}

\section{The construction of \mbox{$\mathrm q$}-tense operators}\label{construction}

In this section we present a canonical construction of Galois 
q-connections and q-tense operators.

By a {\em frame}  it is meant a triple $(S,T,R)$ where $S,T$ are non-void sets 
and $R\subseteq S\times T$. 
If $S=T$, we will write briefly $(T, R)$ for the frame $(T,T,R)$ and we say that $(T, R)$ is a {\em time frame}.  Having 
a q-effect algebra $(E;+,q, d, 0,1)$ and a non-void set $T$, we can produce the direct power
$(E^T;+, q, d, o,j)$ where the operation $+, q$, $d$ and the induced operation  $'$ 
are defined and evaluated on $p,q\in E^T$ componentwise. Moreover, $o, j$ 
are such elements of $E^T$ that $o(t)=0$ and $j(t)=1$
for all $t\in T$. The direct power ${E}^T$ is again a q-effect algebra.

The notion of a time frame allows us to construct  q-tense operators 
on q-effect algebras. 


We have the following   corollary of \cite[Theorem 2]{boturtense}.

\begin{theorem}\label{nvc}
 Let $\mathcal{M}$ be a linearly ordered complete q-effect algebra, 
 $(S, T, R)$ be a frame 
 and $G^{*}$ be a map from 
 ${M}^T$ into ${M}^S$ defined by 
$$G^{*}(p)(s) = \bigwedge\{p(t) \mid t \in T, sRt\},$$ 

\noindent{}for all $p\in M^T$  and $s\in S$.  Then 
$G^{*}$ 
has a left adjoint $P^{*}$ such that 
$(P^{*},G^{*})$ is a Galois q-connection. In this case, 
for all $q\in M^S$  and $t\in T$,  
$$P^{*}(q)(t) = \bigvee\{q(s) \mid s \in T, sRt\}.$$ 
\end{theorem}
\begin{proof} %
Since any linearly ordered complete q-effect algebra is  an MV-algebra we know 
by \cite[Theorem 2]{boturtense} that $(P^{*},G^{*})$ is a Galois connection and that both $G^{*}$ and $P^{*}$ 
preserve the unary operations $q$ and $d$.
\end{proof}

We say that $(P^{*},G^{*})$ is the 
{\em canonical Galois q-connection  induced  by the frame} $(S,T,R)$  
{\em and the q-effect algebra $\mathcal{M}$}.

\begin{corollary}\label{nvccor}
 Let $\mathcal{M}$ be a linearly ordered complete q-effect algebra, 
 $(S, R)$ be a time frame, $G^{*}$ and $H^{*}$ be  maps from 
 ${M}^S$ into ${M}^S$ defined by 
$$\begin{array}{l c l}
G^{*}(p)(s) &=& \bigwedge\{p(t) \mid t \in S, sRt\},\\
H^{*}(p)(s) &=& \bigwedge\{p(t) \mid t \in S, tRs\}
\end{array}
$$ 

\noindent{}for all $p\in M^S$  and $s\in S$.  Then 
$G^*$ ($H^*$) is a q-tense operator on  $\mathcal{M}^S$  which 
has a left adjoint $P^{*}$ ($F^{*}$). In this case, 
for all $q\in M^S$  and $t\in S$,  
$$\begin{array}{l c l}
P^{*}(q)(t)& =& \bigvee\{q(s) \mid s \in S, sRt\},\\
F^{*}(q)(t)& =& \bigvee\{q(s) \mid s \in S, tRs\}.
\end{array}$$ 
Moreover, 
\begin{itemize}
 \item[(i)] if $R$ is reflexive then $G^*(p)\leq p$, $H^*(p)\leq p$, 
$q\leq P^*(q)$ and $q\leq F^*(q)$ hold for all $p, q\in M^S,$
 \item[(ii)] if $R $ is symmetric then $G^*=H^*$ and $P^*=F^*$ hold,
\item[(iii)] if $R$ is transitive then $G^*G^*(p)\geq G^*(p)$, $H^*H^*(p)\geq H^*(p)$, 
$P^*P^*(q)\leq P^*(q)$ and  $F^*F^*(q)\leq P^*(q)$ hold 
for all $p, q\in M^S$.
 \end{itemize} 
\end{corollary}
\begin{proof} %
It follows immediately from Theorem \ref{nvc} that 
$(P^{*},G^{*})$ and $(F^{*},H^{*})$  are Galois q-connections. Similarly 
as in \cite[Theorem 10]{dyn} we have 
$$\bigl( H^{*}(p')'\bigr) (s)=%
    \bigl( \mbox{$\bigwedge$}\{p'(t); t R  s\}\bigr) '=%
    \mbox{$\bigvee$}\{p(t); t R s\}=P^{*}(s).$$
Hence $G^*$ and $H^*$ are q-tense operators.

i) If the relation $R$ is reflexive, then $sR s$ yields $
G^*(p)(s)=\bigwedge\{p(t) \mid t \in S, sRt\}\leq p(s)$ for any $s\in S$ and 
$tR t$ yields $
P^*(q)(t)=\bigvee\{q(s) \mid s \in S, sRt\}\geq q(t)$ for any $t\in S$.
The part for $H^*$ and  $F^*$ we can prove analogously.

ii) If $R$ is symmetric then $G^*(p)(s)=\bigwedge\{p(t) \mid t \in S, sRt\}= %
\bigwedge\{p(t) \mid t \in S, tRs\}=H^*(p)(s)$ 
for any $s\in S$ and any $p\in M^{S}$ which clearly yields $G^*=H^*$. Similarly,  $P^*=F^*$.

iii) If $R $ is transitive, $s\in S$ and  $p\in M^{S}$ 
then $\{p(t)\, |\, sR u\mbox{ and } uR t\}\subseteq\{p(t)\, |\, sR t\}$ for any $u\in S$
and then
\begin{eqnarray*}
G^*G^*(p)(s)&=&\bigwedge\{G^*(p)(u) \mid t \in S, sRu\} =%
\bigwedge\left\{\bigwedge\{p(t) \mid t \in S, uRt\} \mid u \in S, sRu\right\} \\
&=&\bigwedge\{p(t)\, |\, sRu\mbox{ and }  uR t\}
\geq \bigwedge\{p(t) \mid t \in S, sRt\}= G^*(p)(s)
\end{eqnarray*}
holds. Similarly we can prove the remaining inequalities.
\end{proof}

We say that $(\mathcal{M}^{S},G^{*},H^{*})$ is the 
{\em tense  q-effect algebra  induced  by the time frame} $(S,R)$  
{\em and the  q-effect algebra $\mathcal{M}$}.

\section{The representation of \mbox{$\mathrm q$}-tense operators}\label{repres}

The aim of this section is to show that a q-representable q-Jauch-Piron    
q-effect algebra with q-tense operators $G$ and $H$  can 
be represented in a power of the standard q-effect algebra $\mathcal I$, 
where the set of all Jauch-Piron q-states serves as 
a time frame  with a relation defined by means of the pointwise ordering of 
these states on $x$ as well as on $G(x)$ or $H(x)$. It properly means that for 
every q-representable  q-Jauch-Piron   q-effect algebra  with 
q-tense operators there exists a suitable time frame  that  
can be constructed by use of the previously introduced concepts. 

Now we are able to establish our first main result. 

\begin{theorem}\label{mainth}
Let $\mathcal{E}_1=(E_1;+_1,q_1, d_1, 0_1,1_1)$ and 
$\mathcal{E}_2=(E_2;+_2,q_2, d_2,0_2,1_2)$  be  q-effect algebras, 
$f:E_1\to E_2$ and $g:E_2\to E_1$ be  mappings such that 
$(f,g)$ is a Galois q-connection. 
Let $ \mathcal{E}_2$ be a q-representable  q-effect algebra and let $ \mathcal{E}_1$ be a
 q-Jauch-Piron  representable  q-effect algebra, 
$T$ a set of all  q-states on $\mathcal{E}_2$ 
 and $S$ a set of all Jauch-Piron q-states on $\mathcal{E}_1$. 

Further, let $(S, T, R_g)$ be a frame such that 
the relation $R_g\subseteq S\times T$ is defined by 
$$sR_g t\mbox{ if and only if } s(g(x))\leq t(x)\mbox{ for any } x\in E_2.$$

Then $g$ is representable via the  canonical Galois connection $(P^*,G^*)$ 
between complete q-effect algebras $\mathcal I^{S}$  and $\mathcal I^{T}$ induced  by the frame $(S,T,R_g)$  
and the standard q-effect algebra $\mathcal I$, 
$G^*:[0,1]^T \to [0,1]^S$, i.e., the following diagram of mappings commutes:

$$
\begin{diagram}
{E}_2&%
\rTo(2,0)^{\quad\quad{}g}&&{E}_1&\qquad&\qquad&\\
\dTo(0,3)^{i_{{E}_2}^T}&&&\dTo(0,3)_{{i_{{E}_1}^S}}&\qquad&\qquad&\\
[0,1]^T&\rTo(2,0)_{\quad\quad\quad{}G^*}&&[0,1]^S&.\qquad&\qquad&
\end{diagram}
$$
\end{theorem}
\begin{proof} Assume that $x\in E_2$ and $s\in S$. Then 
${i_{ E_1}^S}(g(x))(s)=s(g(x))\leq t(x)$ for all $t\in T$ 
such that $(s,t)\in R_g$. It 
follows that ${i_{ E_1}^S}(g(x))\leq G^{*}({i_{ E_2}^T}(x))$. 

Note that $s\circ g$ is a  Jauch-Piron q-semi-state on $ E_1$. 
Clearly, $s\circ g(1)=s(g(1))=s(1)=1$, $s\circ g$ is order-preserving as a 
composition of two order-preserving mappings, $s\circ g$ preserves unary operations 
$q$ and $d$ since both $s$ and $g$ preserve them. 
Let us check the Jauch-Piron property. 
Namely, assume that $s(g(x))=1= s(g(y))$. Then there is an element 
$z\in E_1$ such that $s(z)=1$, $z\leq g(x)$, $z\leq g(y)$. Since $g$ is a right adjoint 
to the map $f:E_1\to E_2$ we have that $f(z)\leq x$, $f(z)\leq y$. It follows that  
$1=s(z)\leq s(g(f(z))$, i.e., $s(g(f(z))=1$.

 By Theorem 
\ref{prusekapr} we get that 
$$
\begin{array}{r c l}
s\circ g&=&\bigwedge \{ t: {E}_2\to [0,1]\mid \ t\ \text{is an}\ \text{q-state},  
t\geq s\circ g\}\\[0.1cm]
&=&\bigwedge \{ t\in T\mid \ (s, t)\in  R_g\}.
\end{array}$$
This yields that actually ${i_{ E_1}^S}(g(x))= G^{*}({i_{ E_2}^T}(x))$. 
\end{proof}

To get a representation of both $f$ and $g$ from the above theorem we need 
stronger assumptions on $ \mathcal{E}_2$.

\begin{theorem}\label{secondmainth}
Let $\mathcal{E}_1$ and 
$\mathcal{E}_2$  be  q-representable q-Jauch-Piron    q-effect algebras, 
$f:E_1\to E_2$ and $g:E_2\to E_1$ be a mappings such that 
$(f,g)$ is a Galois q-connection. 
Let $T$ be a set of all Jauch-Piron q-states on $\mathcal{E}_2$ 
and $S$ be  a set of all Jauch-Piron q-states on $\mathcal{E}_1$. 

Further, let $(S, T, R_g)$ be a frame such that 
the relation $R_g\subseteq S\times T$ is defined by 
$$sR_g t\mbox{ if and only if } s(g(x))\leq t(x)\mbox{ for any } x\in E_2.$$

Then $f$ and $g$ are representable via the  canonical Galois connection $(P^*,G^*)$ 
between complete q-effect algebras $\mathcal I^{S}$  and $\mathcal I^{T}$ 
 induced  by the frame $(S,T,R_g)$  
and the standard q-effect algebra $\mathcal I$, $P^*:[0,1]^S \to [0,1]^T$, 
$G^*:[0,1]^T \to [0,1]^S$, i.e., the following diagrams of mappings commute:

$$
\begin{array}{c c c}
\begin{diagram}
{E}_1&%
\rTo(2,0)^{\quad\quad{}f}&&{E}_2&\qquad&\qquad&\\
\dTo(0,3)^{i_{{E}_1}^S}&&&\dTo(0,3)_{{i_{{E}_2}^T}}&\qquad&\qquad&\\
[0,1]^S&\rTo(2,0)_{\quad\quad\quad{}P^*}&&[0,1]^T&\qquad&\qquad&
\end{diagram}& 
&\begin{diagram}
{E}_2&%
\rTo(2,0)^{\quad\quad{}g}&&{E}_1&\qquad&\qquad&\\
\dTo(0,3)^{i_{{E}_2}^T}&&&\dTo(0,3)_{{i_{{E}_1}^S}}&\qquad&\qquad&\\
[0,1]^T&\rTo(2,0)_{\quad\quad\quad{}G^*}&&[0,1]^S&.\qquad&\qquad&
\end{diagram}

\end{array}
$$
\end{theorem}
\begin{proof}  Let $(T, S, R_{\overline{f}})$ be a frame such that 
the relation $R_{\overline{f}}\subseteq T\times S$ is defined by 
$$tR_{\overline{f}} s\mbox{ if and only if } t(\overline{f}(w))\leq s(w)\mbox{ for any } w\in E_1.$$

By Lemma \ref{RgRf} we have that 
$$
(s\circ g)(x)\leq t(x)\ \text{for all}\ x\in E_2\ \text{if and only if }\ 
 (t\circ \overline{f})(w)\leq s(w)\ \text{for all}\ w\in E_1.
$$
From Theorem \ref{mainth} it follows that 
$\overline{f}$ is representable via the  canonical Galois connection $(F^*,H^*)$ 
between complete q-effect algebras $\mathcal I^{T}$  and $\mathcal I^{S}$
induced  by the frame $(T,S,R_{\overline{f}})$  
and the standard q-effect algebra ${\mathcal I}$, 
$H^*:[0,1]^T \to [0,1]^S$, i.e., the following diagram of mappings commutes:

$$
\begin{diagram}
{E}_1&%
\rTo(2,0)^{\quad\quad{}\overline{f}}&&{E}_2&\qquad&\qquad&\\
\dTo(0,3)^{i_{{E}_1}^S}&&&\dTo(0,3)_{{i_{{E}_2}^T}}&\qquad&\qquad&\\
[0,1]^S&\rTo(2,0)_{\quad\quad\quad{}H^*}&&[0,1]^T&.\qquad&\qquad&
\end{diagram}
$$
We then have $R_{\overline{f}}=(R_g)^{-1}$ and, for all $y\in E_1$, 
$${i_{{E}_2}^T}(f(y))'={i_{{E}_2}^T}(f(y)')=%
{i_{{E}_2}^T}(\overline{f}(y'))=H^*({i_{{E}_1}^S}(y')).$$ 

Note that 
$$
\begin{array}{r c l}
H^*(q')(t)&=&\bigwedge\{q'(s) \mid t \in T, tR_{\overline{f}}s\}=%
\bigwedge\{q'(s) \mid t \in T, sR_{g}t\}\\[0.2cm]
&=&\left(\bigvee\{q(s) \mid t \in T, sR_{g}t\}\right)'=%
P^{*}(q)(t)'.
\end{array}$$

It follows that ${i_{{E}_2}^T}(f(y)) =H^*({i_{{E}_1}^S}(y'))'=P^{*}({i_{{E}_1}^{S}(y)})$ for all $y\in E_1$. 
\end{proof}

The immediate corollary of Theorem \ref{secondmainth} is the representation theorem 
for q-tense operators.

\begin{theorem}\label{xmaintense}
Let  $\mathcal E$  be a  q-representable   q-Jauch-Piron q-effect algebra  
with an order reflecting set $S$ of  q-states and 
with q-tense operators 
$G$ and $H$. Then $(\mathcal E,G,H)$ can be embedded into the tense 
q-effect algebra $({\mathcal I}^S,G^*,H^*)$ induced by 
the time frame $(S,R_G)$ and the standard q-effect algebra ${\mathcal I}$,  
where $S$ is the set of all  Jauch-Piron q-states 
on $\mathcal E$  and the relation $R_G$ is defined by 
$$sR_G t\mbox{ if and only if } s(G(x))\leq t(x)\mbox{ for any } x\in E.$$
\end{theorem}

In the remaining part we present an  unexpected application of our approach.

\begin{corollary}\label{xmaintense}
Let  $\mathcal M$  be an  MV-algebra  
with an order reflecting set $S$ of  MV-morphisms into the standard MV-algebra  $[0,1]$ and 
with q-tense operators 
$G$ and $H$. Then $(\mathcal M,G,H)$ can be embedded into the tense MV-algebra $([0,1]^S,G^*,H^*)$ induced by 
the time frame $(S,R_G)$ and the standard q-effect algebra ${\mathcal I}$, where the relation $R_G$ is defined by 
$$sR_G t\mbox{ if and only if } s(G(x))\leq t(x)\mbox{ for any } x\in E.$$

In particular,  q-tense operators on $\mathcal M$ are tense operators as defined by \cite{2}. 
\end{corollary}
\begin{proof}
Note that any MV-morphisms from $\mathcal M$ into $[0,1]$  are exactly  Jauch-Piron q-states on $\mathcal M$ 
and that any q-state on $\mathcal M$ is Jauch-Piron. 
Since $G^*$ and $H^*$ are tense operators in the sense of  \cite{2}  
their restrictions to $\mathcal M$,    $G$ and $H$ are also tense. 
\end{proof}

Note also that the same reasoning may be applied to E-tense operators on lattice effect algebras (see \cite{chajapas}).

\section{Concluding remarks}
 We have established  a representation theorem for Galois q-connections and  q-tense operators on q-effect algebras.  
Surprisingly, our q-tense operators are exactly Galois connections connected via negation that preserve only two unary 
operations. In particular, our approach shows that some axioms of tense operators on semisimple MV-algebras  
in \cite{2} were redundant. 

\section*{Acknowledgements}  
This is a pre-print of an article published in Fuzzy Sets and Systems. 
The final authenticated version of the article is available online at: \newline 
https://www.sciencedirect.com/science/article/pii/S0165011415002390.



\end{document}